\documentclass[graybox]{svmult}

\usepackage{type1cm}        
\usepackage{makeidx}         
\usepackage{graphicx}        
\usepackage{multicol}        
\usepackage[bottom]{footmisc}

\usepackage{newtxtext}       %
\usepackage[varvw]{newtxmath}       

\makeindex             

\usepackage{amsmath,amsfonts,amsthm,eucal,mathtools} 
\usepackage{subfig,tikz, color} 
\usetikzlibrary{matrix}

\usepackage[norelsize,ruled,vlined,titlenumbered]{algorithm2e}
\SetKwFor{For}{for}{}{end for}

\usepackage[breakable]{tcolorbox}
\newcommand{\dblock}[1]{\begin{minipage}{\widthof{a #1}}\begin{tcolorbox}[colback=white,colframe=black,right=0cm,left=0cm] #1\end{tcolorbox}\end{minipage}}

\newcommand{\m}[1]{\mathbf{#1}}
\DeclareMathOperator{\supp}{supp}
\newcommand{\cM}{\mathcal{M}}

\newcommand{\NN}{\mathbb{N}}
\newcommand{\defeq}{\coloneq}

\newcommand{\norm}[1]{\left\lVert#1\right\rVert}
\begin{document}

\title*{LR B-spline perspective for RM B-splines:\\
construction and effortless refinements}
\titlerunning{Effortless refinements for RM B-splines}
\author{Francesco Patrizi\orcidID{0000-0002-1836-7906}}
\institute{Francesco Patrizi\at Department of Mathematics, University of Rome Tor Vergata, Via della Ricerca Scientifica 1, 00133, Rome, Italy, \email{patrizi@mat.uniroma2.it}}
%
%
\maketitle

\abstract{
Reachable Minimally supported (RM) B-splines have been recently introduced as a novel B-spline--like basis. They feature local linear independence and admit a fast de Boor--like evaluation algorithm. These properties make them particularly attractive for applications in isogeometric analysis. In this note, we show that automatic mesh refinement procedures can be readily established by observing that RM B-splines are a special case of Locally Refined (LR) B-splines.
}

\section{Introduction}
\label{sec:intro}
The advent of the Isogeometric Analysis (IgA) \cite{iga} has significantly broadened the scope of tensor-product B-splines and NURBS, extending their use from geometric modeling and approximation to numerical simulation. This new approach has introduced several advantages over traditional Finite Element Methods (FEM), including an exact representation of the geometry, independent of the mesh resolution, and a controlled smoothness that allows for continuity up to $C^{p-1}$ for splines of degree $p$, compared to the $C^0$ continuity of FEM. These features lead to more stable solutions, with reduced oscillations, and often a smaller number of degrees of freedom, in some cases up to 90\% fewer than standard FEM for the same accuracy \cite{espen, turbine}. Moreover, IgA provides a direct design-through-analysis workflow that avoids the costly conversion of CAD geometries into analysis-suitable meshes, a process that can consume more than 80\% of the total analysis time in conventional FEM pipelines.

However, the tensor-product structure of classical B-splines and NURBS limits their flexibility: achieving high local accuracy often requires unnecessary refinement in unrelated regions, leading to a large number of redundant degrees of freedom. To overcome this drawback, several spline technologies have been introduced to break the tensor structure while retaining as many of the desirable properties of standard B-splines as possible \cite{ast, pht1, lr, hb, thb, t}. Among these, Locally Refined (LR) B-splines \cite{lr} stand out for their remarkable refinement flexibility. They preserve key B-spline properties and have proven to be highly effective in various approximation tasks, such as data fitting \cite{lrjump, skytt1, skytt2, skytt3, skytt4}.

Nevertheless, the increased freedom in the refinement process comes at a price. The more arbitrary the mesh refinement becomes, the easier it is to generate pathological configurations in which the LR B-spline set loses linear independence. While this is not problematic for many approximation problems it poses a serious obstacle for numerical simulation in the IgA framework. In such cases, the loss of linear independence leads to singular system matrices, complicating the numerical solution process and potentially degrading the overall accuracy.

During the last decade, considerable research has been devoted to understanding and characterizing linear dependencies in LR B-spline sets \cite{lrdep}. Equally important has been the identification of necessary and sufficient conditions ensuring local linear independence, a stronger property that guarantees not only global independence but also enhances the sparsity of the system matrices resulting from the discretizations. Theoretical characterizations have been established \cite{loclinlr}, together with practical refinement strategies that ensure the generated meshes satisfy these conditions. Such strategies either impose a certain structures on the mesh, as in the refinements introduced in \cite{hlr, eg}, or apply local corrective meshline estensions to restore local linear independence when it is lost, as it is done in \cite{n2s2}.

This latter approach is adopted as well in the definition of Reachable Minimally supported (RM) B-splines  \cite{rm1}. RM B-splines form a particular subclass of LR B-splines that are guaranteed to be locally linearly independent by construction, due to the corrective meshline extension procedure designed to recover this property. Furthermore, RM B-splines introduce an additional simplification to the general LR B-spline framework by enforcing a relationship between the meshline multiplicities and the polynomial degree, inspired by the construction of $C^s$-smooth PHT-splines \cite{pht2}.  This constraint reduces the number of possible local configurations of supports, resulting in a mathematically simpler and easier-to-analyze framework. This simplification has been exploited to develop a compact data structure and a de Boor--like evaluation algorithm \cite{rm2}. More recently, stronger criteria for corrective meshline extesions have been proposed to ensure not only local linear independence but also the achievement of certain mesh quality parameters after refinement \cite{rm3} and to improve the mesh structure, reducing propagation.

Although RM B-splines are a specific instance of LR B-splines, this relationship has been not exploited in the design of automatic refinements. In this work, we show that the refinement procedures developed for locally linearly independent LR B-splines can be directly applied to generate RM B-splines. We therefore propose three ready-to-use refinement strategies for RM B-splines, derived from the existing LR refinement methods. Moreover, we show that RM B-splines can be easily incorporated into the LR B-spline framework, providing a straightforward and efficient way to integrate RM B-splines into existing LR implementations. Finally, we carry out numerical comparisons between RM B-splines and LR B-splines of maximal smoothness, when imposing either the same degree or the same regularity. More precisely we investigate the cardinalities of the two sets when varying these parameters and the achieved accuracy in numerical simulations.

The remainder of this paper is organized as follows. In Section \ref{sec:lr}, we provide a brief introduction to the LR B-spline framework, focusing on the bivariate case. Section \ref{sec:rm} introduces the RM B-splines from the LR B-spline perspective, building upon the definitions and procedures established in the previous section. In Section \ref{sec:ref}, we explain why and how the refinement strategies developed for locally linearly independent LR B-splines can also be used to generate RM B-splines, and we present three such refinement schemes. In Section \ref{sec:imp} we show how RM B-splines can be naturally embedded into the LR framework and its implementations and in Section \ref{sec:num} we make numerical comparisions between RM B-splines and maximally smooth LR B-splines, both in terms of set cardinalities and accuracy in numerical simulations. Finally, we draw the conclusions in Section \ref{sec:con}. 

\section{Locally Refined B-splines}
\label{sec:lr}
In this section, we briefly recall the definition of \emph{Locally Refined (LR) B-splines} in the bivariate setting. For the general definition in $\mathbb{R}^n$, we refer to the seminal paper \cite{lr}. We also assume that the reader is familiar with tensor product B-splines and their basic properties. Concise but comprehensive introductions to this topic can be found in \cite{manni1} and \cite{manni2}, while more extensive treatments are available in the classical references \cite{deboor} and \cite{schu}. Let us start by introducing some basic definitions.


\begin{definition}[Meshline multiplicity]
Let $\mathbf{p} = (p_1, p_2)$ be a bidegree and let $\mathcal{M}$ be a tensor-product mesh.
Each meshline $\gamma$ of $\mathcal{M}$ is assigned an integer multiplicity
$$
1 \le \mu(\gamma) \le p_k + 1,
$$
where $k = 1$ for vertical meshlines and $k = 2$ for horizontal meshlines. The tensor-product mesh $\cM$ is called open if the boundary meshlines have multiplicity $p_k + 1$ with $k = 1$ (vertical edges) and $k = 2$ (horizontal edges).
\end{definition}

\begin{definition}[Local tensor mesh \& B-spline support]
Let $B$ be a tensor-product B-spline defined on $\mathcal{M}$.
The local knot vectors of $B$ determine a local tensor mesh
$$
\mathcal{M}_B \subseteq \mathcal{M},
$$
consisting of the $p_1 + 2$ consecutive vertical meshlines and the $p_2 + 2$ consecutive horizontal meshlines (counting multiplicities) associated with the local knot vectors.
The region enclosed by $\mathcal{M}_B$ coincides with the support of $B$, denoted by $\supp B$.
\end{definition}

\begin{definition}[Local meshline multiplicity]
Let $\gamma$ be a meshline of $\mathcal{M}_B$. The local meshline multiplicity of $\gamma$ with respect to $B$ is defined as the number of times the parameter value of $\gamma$ appears in the local knot vectors of $B$. We denote this by $\mu_B(\gamma)$.
\end{definition}

Suppose, for example, that $\gamma_B = \gamma \cap \mathcal{M}_B$ is a vertical local meshline of the form $\gamma_B = \{x_B^j\} \times [y_B^1,\, y_B^{p_2+2}]$.
Then $\mu_B(\gamma_B)$ is equal to the multiplicity of  $x_B^j$ in the local knot vector of $B$ in the first parametric direction.

\begin{remark}
Clearly, $\mu_B(\gamma_B) \le \mu(\gamma)$, since a local multiplicity cannot exceed the global multiplicity of the meshline from which it is derived.
\end{remark}

In Figure \ref{fig:tensorex} we display a tensor mesh and some local tensor meshes of B-splines of bidegre $(2, 2)$ defined on it.
\begin{figure}
\centering
\subfloat[]{
\begin{tikzpicture}
\draw[step = 1, thick] (0, 0) grid (4, 4);
\draw[xshift = - 0.05cm, thick] (1, 0) -- (1, 4);
\draw[xshift = 0.05cm, thick] (1, 0) -- (1, 4);
\draw (1, 4) node[above]{$\gamma$};
\end{tikzpicture}
}\quad
\subfloat[]{
\begin{tikzpicture}
\draw[step = 1, thick] (0, 0) grid (4, 4);
\draw[xshift = - 0.05cm, thick] (1, 0) -- (1, 4);
\draw[xshift = 0.05cm, thick] (1, 0) -- (1, 4);
\draw (1, 4) node[above]{$\gamma$};
\filldraw[fill = red!50!white, draw = red, thick] (0.95, 0) -- (2, 0) -- (2, 3) -- (0.95, 3) -- cycle;
\draw[red, xshift = 0.00cm, thick] (1, 0) -- (1, 3);
\draw[red, xshift = 0.05cm, thick] (1, 0) -- (1, 3);
\draw[red, thick] (0.95, 1) -- (2, 1);
\draw[red, thick] (0.95, 2) -- (2, 2);
\draw[red] (1.5, 3) node[above]{$B_1$};
\end{tikzpicture}
}\quad
\subfloat[]{
\begin{tikzpicture}
\draw[step = 1, thick] (0, 0) grid (4, 4);
\draw[xshift = - 0.05cm, thick] (1, 0) -- (1, 4);
\draw[xshift = 0.05cm, thick] (1, 0) -- (1, 4);
\draw (1, 4) node[above]{$\gamma$};
\filldraw[fill = green!50!black!50!white, draw = green!50!black, thick] (1, 0) -- (3, 0) -- (3, 3) -- (1, 3) -- cycle;
\draw[green!50!black, xshift = 0.05cm, thick] (1, 0) -- (1, 3);
\draw[green!50!black, thick] (1, 1) -- (3, 1);
\draw[green!50!black, thick] (1, 2) -- (3, 2);
\draw[green!50!black, thick] (2, 0) -- (2, 3);
\draw[green!50!black, thick] (2, 3) node[above]{$B_2$};
\end{tikzpicture}
}\quad
\subfloat[]{
\begin{tikzpicture}
\draw[step = 1, thick] (0, 0) grid (4, 4);
\draw[xshift = - 0.05cm, thick] (1, 0) -- (1, 4);
\draw[xshift = 0.05cm, thick] (1, 0) -- (1, 4);
\draw (1, 4) node[above]{$\gamma$};
\filldraw[fill = cyan!80!black!50!white, draw = cyan!80!black, thick] (1.05, 0) -- (4, 0) -- (4, 3) -- (1.05, 3) -- cycle;
\draw[cyan!80!black, thick] (1.05, 1) -- (4, 1);
\draw[cyan!80!black, thick] (1.05, 2) -- (4, 2);
\draw[cyan!80!black, thick] (2, 0) -- (2, 3);
\draw[cyan!80!black, thick] (3, 0) -- (3, 3);
\draw[cyan!80!black, thick] (2.5, 3) node[above]{$B_3$};
\end{tikzpicture}
}
\caption{Example of tensor product mesh and some tensor product B-splines on it. In (a) the mesh $\cM$. Given the degree $\m{p} = (2, 2)$, the vertical meshline $\gamma$ has multiplicity $\mu(\gamma) = 3$, i.e., the maximal multiplicity allowed for $\m{p}$. Such multiplicity is emphasized in the figures by drawing $\gamma$ with a triple line. In (b)--(d) we consider three different biquadratic B-splines $B_1, B_2$ and $B_3$ on $\cM$.  The highlighted regions correspond to their supports and their local tensor meshes are also coloured. In particular (part of) $\gamma$ is in all of them, as $\gamma_{B_1} = \gamma_{B_2} = \gamma_{B_3}$. However, it is here considered with different multiplicities: $\mu_{B_1}(\gamma_{B_1}) = 3, \mu_{B_2}(\gamma_{B_2}) = 2, \mu_{B_3}(\gamma_{B_3}) = 1$.}\label{fig:tensorex}
\end{figure}
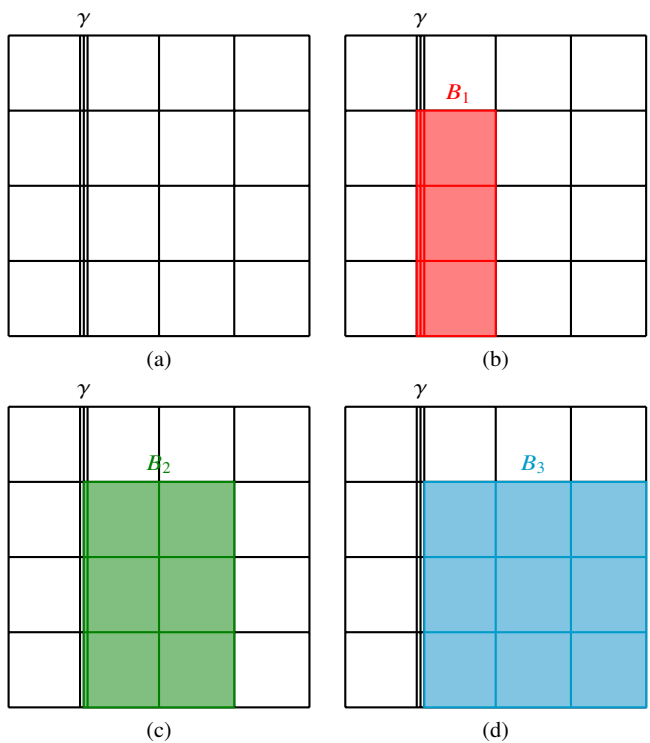

The collection of B-splines on a tensor mesh may be enriched in two ways:
\begin{enumerate}
    \item by inserting a new meshline into $\mathcal{M}$ that intersects the support of existing B-splines, or
    \item by increasing the multiplicity $\mu(\gamma)$ of an existing meshline $\gamma \in \mathcal{M}$.
\end{enumerate}
In either case, some of the existing B-splines will lose the \textbf{minimal support property}: their local tensor meshes $\mathcal{M}_B$ can no longer be obtained by taking consecutive $p_1 + 2$ vertical and $p_2 + 2$ horizontal meshlines (counting multiplicities). In other words, certain meshlines of the refined mesh are ``skipped'' when forming the local tensor meshes of these B-splines. In terms of their local knot vectors, the skipped meshlines of the mesh correspond to missed knots in the knot sequences.
However, by the knot insertion procedure \cite[Section 3.5]{manni1}, any B-spline defined on a given knot vector can be expressed as a linear combination of two B-splines defined on a refined knot vector when a new knot is inserted. These two new B-splines are defined on local knot vectors corresponding to the first and last $p_k + 2$ consecutive knots in the extended vector of length $p_k + 3$, which includes the new knot. Consequently, the space can be enriched by replacing each non-minimally supported B-spline with the newly generated B-splines, thereby recovering the minimal support property:  
each local tensor mesh consists of consecutive $p_1 + 2$ vertical and $p_2 + 2$ horizontal meshlines (counting multiplicities) of the underlying mesh. We are ready for the definitions of LR B-splines and LR meshes.

\begin{definition}[Locally Refined (LR) B-splines]
LR B-splines are defined recursively as follows:
\begin{itemize}
    \item If the underlying mesh is a tensor-product mesh, LR B-splines coincide with the tensor-product B-splines.
    \item Otherwise, LR B-splines are obtained by recursively applying the knot insertion procedure to previously defined LR B-splines on a coarser mesh, ensuring that all resulting B-splines satisfy the minimal support property.
\end{itemize}
\end{definition}

\begin{definition}[Locally Refined (LR) mesh]
An LR mesh is either a tensor-product mesh or a mesh generated by inserting a new meshline that traverses the support of at least one LR B-spline, or by increasing the multiplicity of an existing meshline,  subject to the constraint that each multiplicity remains at most $p_k + 1$ for $k \in \{1, 2\}$.
\end{definition}

\begin{remark}
The insertion of a new meshline may trigger the knot insertion procedure multiple times, since the two newly generated B-splines may themselves lack of minimal support due to the presence of other existing meshlines, see Figure~\ref{fig:lrexample} for an illustrative example.
\end{remark}
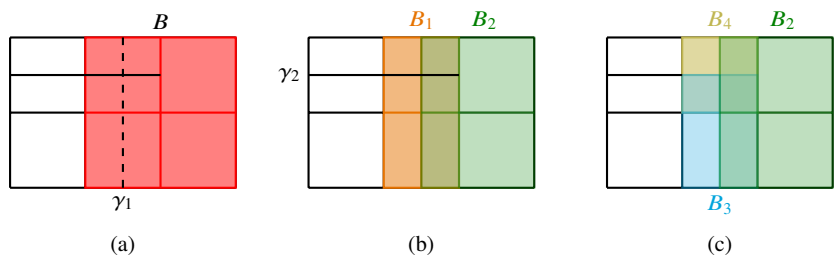
\begin{figure}
\centering
\subfloat[]{
\begin{tikzpicture}
\draw[step = 1, thick] (0, 0) grid (3, 2);
\filldraw[fill = red!50!white, draw = red, thick] (1, 0) -- (3, 0) -- (3, 2) -- (1, 2) -- cycle;
\draw[red, thick] (1, 1) -- (3, 1);
\draw[red, thick] (2, 0) -- (2, 2); 
\draw[thick] (0, 1.5) -- (2, 1.5);
\draw[thick, dashed] (1.5, 0) -- (1.5, 2);
\draw (1.5, 0) node[below]{$\gamma_1$};
\draw[white] (2, 0) node[below]{$B_3$};
\draw (2, 2) node[above]{$B$};
\end{tikzpicture}
}\quad
\subfloat[]{
\begin{tikzpicture}
\draw[step = 1, thick] (0, 0) grid (3, 2);
\filldraw[fill = orange!90!black!50!white, draw = orange!90!black, thick] (1, 0) -- (2, 0) -- (2, 2) -- (1, 2) -- cycle;
\draw[orange!90!black, thick] (1, 1) -- (2, 1);
\draw[orange!90!black, thick] (1.5, 0) -- (1.5, 2); 
\filldraw[fill = green!50!black!50!white, draw = green!50!black, thick, opacity = 0.5] (1.5, 0) -- (3, 0) -- (3, 2) -- (1.5, 2) -- cycle;
\draw[green!50!black, thick, opacity = 0.5] (1.5, 1) -- (3, 1);
\draw[green!50!black, thick, opacity = 0.5] (2, 0) -- (2, 2);
\draw[thick] (0, 1.5) -- (2, 1.5);
\draw[white] (1.5, 0) node[below]{$\gamma_1$};
\draw[white] (2, 0) node[below]{$B_3$};
\draw (0, 1.5) node[left]{$\gamma_2$};
\draw[white] (3, 1.5) node[right]{$\gamma_2$};
\draw[orange!90!black] (1.5, 2) node[above]{$B_1$};
\draw[green!50!black] (2.333, 2) node[above]{$B_2$};
\end{tikzpicture}
}\quad
\subfloat[]{
\begin{tikzpicture}
\draw[step = 1, thick] (0, 0) grid (3, 2);
\filldraw[fill = yellow!70!black!50!white, draw = yellow!70!black, thick] (1, 1) -- (2, 1) -- (2, 2) -- (1, 2) -- cycle;
\draw[yellow!70!black, thick, opacity = 0.5] (1, 1.5) -- (2, 1.5);
\draw[yellow!70!black, thick, opacity = 0.5] (1.5, 1) -- (1.5, 2); 
\filldraw[fill = cyan!90!black!50!white, draw = cyan!90!black, opacity = 0.5, thick] (1, 0) -- (2, 0) -- (2, 1.5) -- (1, 1.5) -- cycle;
\draw[cyan!90!black, thick, opacity = 0.5] (1, 1) -- (2, 1);
\draw[cyan!90!black, thick, opacity = 0.5] (1.5, 0) -- (1.5, 1.5);
\filldraw[fill = green!50!black!50!white, draw = green!50!black, thick, opacity = 0.5] (1.5, 0) -- (3, 0) -- (3, 2) -- (1.5, 2) -- cycle;
\draw[green!50!black, thick, opacity = 0.5] (1.5, 1) -- (3, 1);
\draw[green!50!black, thick, opacity = 0.5] (2, 0) -- (2, 2);
\draw[thick] (0, 1.5) -- (1, 1.5);
\draw[white] (1.5, 0) node[below]{$\gamma_1$};
\draw[yellow!70!black] (1.5, 2) node[above]{$B_4$};
\draw[green!50!black] (2.333, 2) node[above]{$B_2$};
\draw[cyan!90!black] (1.5, 0) node[below]{$B_3$};
\end{tikzpicture}
}
\caption{Mesh refinement and secondary split of the generated B-splines to recover the minimal support property. Let $\m{p} = (1, 1)$. In figure (a) we see a bilinear B-spline $B$ which is refined by the new meshline $\gamma_1$ (dashed). $B$ does not have minimal support on the refined mesh. We replace it with the two B-splines $B_1, B_2$ pictured in figure (b), involved in the knot insertion relation along the first direction. However, $B_1$ has not minimal support on the new mesh as well as its support is traversed by a meshline $\gamma_2$ on the mesh. Therefore, also $B_1$ is replaced using the knot insertion procedure. The final set of B-spline $\{B_2, B_3, B_4\}$ generated from the refinement of $B$ is illustrated in figure (c).}\label{fig:lrexample}
\end{figure}

LR B-splines have been successfully employed in geometric design, data modeling, and simulation \cite{lriga3, lriga2, lriga1, lriga5, lriga4} within the framework of Isogeometric Analysis. However, their application in numerical simulation is hindered by the potential occurrence of linear dependencies among the LR B-splines. In such cases, the collection of LR B-splines forms only a generator set rather than a basis. Thus, the linear systems arising from the discretization of differential problems become singular, which complicates the numerical solution and may adversely affect accuracy.
However, under certain conditions LR B-splines are even locally linearly independent.
\begin{definition}[Local linear independence]
A collection of functions is said to be locally linearly independent if it is linearly independent on every open subset of the domain.
\end{definition}
Local linear independence implies global linear independence, allowing the LR B-spline set to be regarded as a basis of the space they span in these cases. For LR B-splines, local linear independence is equivalent to linear independence on each mesh box (cell). Thus, verifying linear independence on individual cells is equivalent to guarantee local linear independence on arbitrary open sets, as proved by the following  characterization \cite[Theorem 4]{hlr}.

\begin{theorem}[Characterization of Local Linear Independence with Non-Overloading]
LR B-splines are locally linearly independent if and only if every mesh cell is non-overloaded, meaning that each cell is covered by exactly $(p_1 + 1)(p_2 + 1)$ LR B-spline supports. 
\end{theorem}

\begin{remark}
The number $(p_1+1)(p_2+1)$ is equal to the dimension of the tensor-product polynomial space of bi-degree $\m{p} = (p_1,p_2)$ over a single cell.
Thus:
\begin{itemize}
    \item a non-overloaded cell has exactly the number of LR B-splines required to span the polynomial space,
    \item an overloaded cell is a cell covered by more than $(p_1+1)(p_2+1)$ LR B-splines, which necessarily exhibits linear dependence on that cell and therefore fail to satisfy local linear independence.
\end{itemize}
If the LR B-spline construction starts from an open tensor-product mesh then every cell is initially non-overloaded. Then, mesh refinements can only increase the number of B-splines with support over a cell. Thus, the number of functions covering each new cell can only exceed, but never fall below, $(p_1 + 1)(p_2 + 1)$.
\end{remark}
The non-overloading condition will play a crucial role in the next section, where we define the Reachable Minimally supported (RM) B-splines.

\section{Reachable Minimally supported B-splines}
\label{sec:rm}
In this section, we introduce the \emph{Reachable Minimally supported (RM) B-splines}. As they form a special subclass of LR B-splines, we will rely on the notation and concepts introduced in Section \ref{sec:lr}. Nevertheless, it is also possible to define RM B-splines independently, without reference to the LR framework, as done in \cite{rm1}. Here, we adopt an LR-based formulation to emphasize the connection between the two constructions.

Let $\mathcal{M}$ be a tensor-product mesh and let the bidegree $\m{p} = (p, p)$ with $p \defeq 2s + 1$ and $s \in \mathbb{N} = \{0, 1, \ldots\}$. We assign the meshline multiplicities as follows:
\begin{itemize}
    \item every internal meshline of $\mathcal{M}$ receives multiplicity $s+1$,
    \item every boundary meshline receives multiplicity $p+1$.
\end{itemize}
The tensor-product B-splines of bidegree $\mathbf{p}$ defined on $\mathcal{M}$ are grouped into collections of $(s+1)^2$ functions, called B-spline systems.

\begin{definition}[B-spline system]
Two B-splines $B_1$ and $B_2$ belong to the same B-spline system if and only if 
$$
\supp B_1 = \supp B_2.
$$
\end{definition}

\begin{remark}
Because of the meshline multiplicities, exactly $(s+1)^2$ tensor-product B-splines share each support.  
Each support consists of four cells, except near the boundary where fewer cells occur due to the higher boundary multiplicities.
\end{remark}

An illustrative example of B-spline system is provided in Figure \ref{fig:bsystem}.
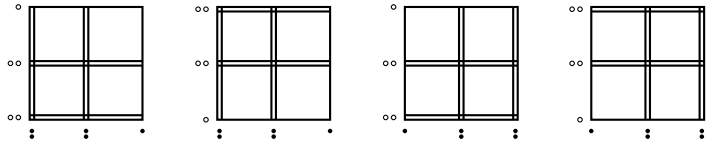
\begin{figure}
\centering
\subfloat{
\begin{tikzpicture}[scale = 1.5]
\draw[thick] (0, 0) rectangle (1, 1);
\draw[thick, xshift = - 0.02cm] (0.5, 0) -- (0.5, 1);
\draw[thick, xshift = 0.02cm] (0.5, 0) -- (0.5, 1);
\draw[thick, yshift = -0.02cm] (0, 0.5) -- (1, 0.5);
\draw[thick, yshift = 0.02cm] (0, 0.5) -- (1, 0.5);

\draw[thick, xshift = 0.04cm] (0, 0) -- (0, 1);
\draw[thick, yshift = 0.04cm] (0, 0) -- (1, 0);

\fill[xshift = 0.02cm, yshift = -0.1cm] (0, 0) circle (0.02);
\fill[xshift = 0.02cm, yshift = -0.15cm] (0, 0) circle (0.02);
\fill[yshift = -0.1cm] (0.5, 0) circle (0.02);
\fill[yshift = -0.15cm] (0.5, 0) circle (0.02);
\fill[yshift = -0.1cm] (1, 0) circle (0.02);

\draw[yshift = 0.02cm, xshift = -0.1cm] (0, 0) circle (0.02);
\draw[yshift = 0.02cm, xshift = -0.17cm] (0, 0) circle (0.02);
\draw[xshift = -0.1cm] (0, 0.5) circle (0.02);
\draw[xshift = -0.17cm] (0, 0.5) circle (0.02);
\draw[xshift = -0.1cm] (0, 1) circle (0.02);
\end{tikzpicture}}\qquad
\subfloat{
\begin{tikzpicture}[scale = 1.5]
\draw[thick] (0, 0) rectangle (1, 1);
\draw[thick, xshift = - 0.02cm] (0.5, 0) -- (0.5, 1);
\draw[thick, xshift = 0.02cm] (0.5, 0) -- (0.5, 1);
\draw[thick, yshift = -0.02cm] (0, 0.5) -- (1, 0.5);
\draw[thick, yshift = 0.02cm] (0, 0.5) -- (1, 0.5);

\draw[thick, xshift = 0.04cm] (0, 0) -- (0, 1);
\draw[thick, yshift = -0.04cm] (0, 1) -- (1, 1);

\fill[xshift = 0.02cm, yshift = -0.1cm] (0, 0) circle (0.02);
\fill[xshift = 0.02cm, yshift = -0.15cm] (0, 0) circle (0.02);
\fill[yshift = -0.1cm] (0.5, 0) circle (0.02);
\fill[yshift = -0.15cm] (0.5, 0) circle (0.02);
\fill[yshift = -0.1cm] (1, 0) circle (0.02);

\draw[yshift = -0.02cm, xshift = -0.1cm] (0, 1) circle (0.02);
\draw[yshift = -0.02cm, xshift = -0.17cm] (0, 1) circle (0.02);
\draw[xshift = -0.1cm] (0, 0.5) circle (0.02);
\draw[xshift = -0.17cm] (0, 0.5) circle (0.02);
\draw[xshift = -0.1cm] (0, 0) circle (0.02);
\end{tikzpicture}}\qquad
\subfloat{
\begin{tikzpicture}[scale = 1.5]
\draw[thick] (0, 0) rectangle (1, 1);
\draw[thick, xshift = - 0.02cm] (0.5, 0) -- (0.5, 1);
\draw[thick, xshift = 0.02cm] (0.5, 0) -- (0.5, 1);
\draw[thick, yshift = -0.02cm] (0, 0.5) -- (1, 0.5);
\draw[thick, yshift = 0.02cm] (0, 0.5) -- (1, 0.5);

\draw[thick, xshift = -0.04cm] (1, 0) -- (1, 1);
\draw[thick, yshift = 0.04cm] (0, 0) -- (1, 0);

\fill[xshift = -0.02cm, yshift = -0.1cm] (1, 0) circle (0.02);
\fill[xshift = -0.02cm, yshift = -0.15cm] (1, 0) circle (0.02);
\fill[yshift = -0.1cm] (0.5, 0) circle (0.02);
\fill[yshift = -0.15cm] (0.5, 0) circle (0.02);
\fill[yshift = -0.1cm] (0, 0) circle (0.02);

\draw[yshift = 0.02cm, xshift = -0.1cm] (0, 0) circle (0.02);
\draw[yshift = 0.02cm, xshift = -0.17cm] (0, 0) circle (0.02);
\draw[xshift = -0.1cm] (0, 0.5) circle (0.02);
\draw[xshift = -0.17cm] (0, 0.5) circle (0.02);
\draw[xshift = -0.1cm] (0, 1) circle (0.02);
\end{tikzpicture}}\qquad
\subfloat{
\begin{tikzpicture}[scale = 1.5]
\draw[thick] (0, 0) rectangle (1, 1);
\draw[thick, xshift = - 0.02cm] (0.5, 0) -- (0.5, 1);
\draw[thick, xshift = 0.02cm] (0.5, 0) -- (0.5, 1);
\draw[thick, yshift = -0.02cm] (0, 0.5) -- (1, 0.5);
\draw[thick, yshift = 0.02cm] (0, 0.5) -- (1, 0.5);

\draw[thick, xshift = -0.04cm] (1, 0) -- (1, 1);
\draw[thick, yshift = -0.04cm] (0, 1) -- (1, 1);

\fill[xshift = -0.02cm, yshift = -0.1cm] (1, 0) circle (0.02);
\fill[xshift = -0.02cm, yshift = -0.15cm] (1, 0) circle (0.02);
\fill[yshift = -0.1cm] (0.5, 0) circle (0.02);
\fill[yshift = -0.15cm] (0.5, 0) circle (0.02);
\fill[yshift = -0.1cm] (0, 0) circle (0.02);

\draw[yshift = -0.02cm, xshift = -0.1cm] (0, 1) circle (0.02);
\draw[yshift = -0.02cm, xshift = -0.17cm] (0, 1) circle (0.02);
\draw[xshift = -0.1cm] (0, 0.5) circle (0.02);
\draw[xshift = -0.17cm] (0, 0.5) circle (0.02);
\draw[xshift = -0.1cm] (0, 0) circle (0.02);
\end{tikzpicture}}
\caption{Example of B-spline system in the interior of the mesh for $s = 1$. All the internal meshlines of the mesh have multiplicity $s + 1 = 2$. In the figures we show the local knot vectors and local tensor meshes of the B-splines belonging to the same system. Their local knot vectors, composed of $p + 2 = 5$ knots, in the first and second directions have been illustrated with full and empty dots respectively. Furthermore, double dots correspond to knots of multiplicity $2$ in the local knot vector. Such local knot vectors generate local tensor meshes, covering the support of the associated B-splines. Double dots correspond to double lines to emphasize the higher local meshline multiplicity considered. Although the 4 B-splines are all different, they have same support and hence belong to the same system.}\label{fig:bsystem}
\end{figure}

\begin{remark}
If $s=0$, each system consists of a single bilinear tensor-product B-spline. 
\end{remark}

For any $s \in \mathbb{N}$, every mesh cell belongs to exactly four B-spline systems, due to the tensor product structure. 
During mesh refinement, only new meshlines traversing the supports of existing B-splines are allowed, as it always happens in the LR B-spline setting. Each new meshline is assigned multiplicity $s + 1$. Since the B-splines within a given system share the same support, every insertion triggers the refinement of entire B-spline systems, i.e., of $(s +1)^2$ B-splines, rather than of single functions. After insertion, all functions belonging to the traversed systems are replaced by new functions with minimal support on the refined mesh, in accordance with the standard knot insertion procedure of the LR B-spline framework, forming new B-spline systems.

However, to qualify as RM B-splines, the generated LR B-splines must also preserve the property that each mesh cell remains contained in the supports of the B-splines belonging to exactly four B-spline systems after every insertion. This ensures that the resulting set of functions maintains local linear independence, since each cell is then covered by $4(s + 1)^2 = (p + 1)^2$ function supports, precisely the number required to span the bivariate polynomial space over that cell. In other words, no cell shall be overloaded after the refinement.

In practice, when a meshline is inserted for refinement and new functions are created, overloaded cells are identified and the inserted meshline is extended so that it traverses all B-spline systems containing those cells. Overloading is then checked again after each extension. This process is guaranteed to terminate after finitely many steps and ensures local linear independence, see \cite[Section 5]{rm1} for the case $s = 0$ and \cite[Section 6.4]{rm1} for the general case. 
After all the needed extensions have been performed following the insertion of a new meshline, the resulting B-splines satisfy the following properties:
\begin{itemize}
\item \textbf{R}eachability: they are obtained from an initial set of tensor product B-splines through successive knot insertions, as always in the LR B-spline framework,
\item \textbf{M}inimal support: no meshline of the underlying mesh traverses their support without being part of their local tensor mesh, again as always in the LR B-spline setting,
\item form a \textbf{B}asis: they are locally linearly independent on the mesh, which implies (global) linear independence.
\end{itemize}

This gives rise to the terminology \emph{RM B-splines}. However, RM B-splines can also be interpreted as a specific refinement strategy for LR B-splines that guarantees local linear independence for LR B-splines of degree $p = 2s + 1$ defined on LR meshes with meshline multiplicities adjusted as described above. From this perspective, the RM refinement procedure can be viewed as an LR refinement process in which newly inserted meshlines are systematically extended to ensure the local linear independence property.

On the other hand, locally linearly independent LR B-splines of degree $p = 2s + 1$ on LR meshes with such meshline multiplicity distribution, namely, the RM B-splines, can also be obtained with other approaches, relaying on the existing results for LR B-splines of arbitrary degree. The next section will present automatic refinement strategies that guarantee local linear independence of the created B-splines without requiring explicit checks for overloaded cells and hence further meshline extensions. We conclude this introductory section on RM B-splines with a few final remarks.

\begin{remark}
We recall that in the introductory paper on RM B-splines \cite{rm1}, the domain $\Omega = [a, b]^2$ is extended with a \emph{phantom zone} to $[a-1, b+1]^2$, and the initial tensor mesh is defined over this enlarged domain. This extension avoids the use of open meshes and hence allows assigning a uniform multiplicity of $s + 1$ to all meshlines, including those at the boundaries. Such a choice simplifies the construction of the B-spline system, as each system then corresponds one-to-one to a block of four cells in the tensor mesh over the extended domain. Refinement starts, of course, marking only cells within $\Omega$.  
If, on the other hand, we work directly on the original domain with open meshes, the boundary B-spline systems cover smaller regions of the mesh because some boundary cells are collapsed and result in increased multiplicities of the boundary meshlines. Both constructions, however, are equivalent when restricted to $\Omega$.
\end{remark}

\begin{remark}
In some works \cite{rm1, rm2}, the new meshlines introduced during refinement are placed at dyadic rational parameter values. This was initially proposed to improve the shape quality of the newly generated mesh cells. Nevertheless, this constraint is not necessary for the general definition of RM B-splines.
\end{remark}

\begin{remark}
As previously discussed, local linear independence is ensured after inserting a new meshline by extending it. Specifically, if a cell becomes overloaded, the meshline is extended to traverse all B-spline systems that contain that cell. In certain cases, however, extending the meshline through all such supports may not be strictly necessary to achieve local linear independence. Extending it through only a subset of these systems might be sufficient, resulting in a more localized refinement and reduced meshline propagation beyond the region marked for refinement. This procedure has been adopted in the algorithm proposed in \cite{n2s2} for enforcing local linear independence in LR B-splines of arbitrary degree and maximal smoothness.
\end{remark}

\begin{remark}
Stricter refinement rules have also been proposed for generating RM B-splines. In addition to ensuring local linear independence, new meshlines can be inserted and extended to meet specific mesh quality parameters \cite{rm3} or to satisfy the so-called Local Semi-Regularity (LSR) condition \cite{rm1}. These optional constraint, implemented via flags in the algorithms, are designed to limit refinement propagation outside the marked region and to improve the overall mesh structure. However, they are not mandatory for RM B-spline generation. The procedure remains essentially the same as for standard RM B-splines, except that these additional conditions are checked in place of the overloading test. In fact, satisfying either the mesh quality parameters or the LSR condition automatically guarantees local linear independence. Therefore, verifying one of these conditions is sufficient for the resulting LR B-splines to be classified as RM B-splines.
\end{remark}

\begin{remark}
In \cite{rm1}, the bilinear RM B-splines are introduced without making use of the LR B-spline framework.  
The knot insertion procedure employed to restore the minimal support property is there called \emph{fixed point refinement operator}, and LR meshes built with respect to bidegree $\mathbf{p} = (1,1)$ are named as \emph{T-meshes with Associated Bilinear B-splines (TAB)}.
The simplified nature of the bilinear case enables an explicit identification of all the possible local displacement configurations, in the topological sense, of the B-spline supports covering a chosen cell under the RM B-spline generation algorithm.  
In the general situation (without the LSR condition), 385 distinct configurations have been classified numerically. When the LSR condition is imposed, this number reduces to 49. 
Non-overloading is verified in each configuration, and therefore the local linear independence of bilinear RM B-splines is established by exhaustion.
\end{remark}

\section{Effortless refinements and generation of RM B-splines}
\label{sec:ref}
The RM B-splines form a special subclass of LR B-splines of degree $\m{p} = (p, p)$ with $p = 2s + 1$, defined on peculiar LR meshes that we shall call \emph{admissible for RM B-splines}. 
\begin{definition}[Admissible LR mesh for RM B-splines]
Let $\mathbf{p} = (p,p)$ with $p = 2s+1$ and $s \in \mathbb{N}$.  
An LR mesh $\mathcal{M}$ is called admissible for RM B-splines of degree $\mathbf{p}$ if the following conditions hold:
\begin{enumerate}
    \item Every internal meshline has multiplicity $s+1$, and every boundary meshline has multiplicity $2s+2 = p + 1$, i.e., the mesh is open.
    \item The meshlines are positioned in such a manner to ensure local linear independence of the LR B-splines of degree $\mathbf{p}$ defined on $\mathcal{M}$.
\end{enumerate}
\end{definition}
Although this characterization has always been known \cite{rm1}, it has never been exploited for the refinement and construction of RM B-splines. The adopted methodology in the previous works relies on the refinement procedure outlined in Section \ref{sec:rm}, which requires verifying the admissibility of each newly inserted meshline and, when necessary, performing corrective extensions. In this section, we show that RM B-splines can be instead constructed directly from sets of locally linearly independent bilinear LR B-splines, simply by increasing the multiplicities of the meshlines. Thereby, any mesh refinement strategy for LR B-splines that guarantees the local linear independence, applied in the bilinear case, can be safely employed to generate RM B-splines, without the need to check overloading  and apply corrections during refinement.

The following result provides the key ingredient for this generation procedure.
\begin{proposition}\label{prop}
If an LR mesh is admissible for RM B-splines of degree $\m{1} = (1, 1)$, then it shall be admissible for RM B-splines of any degree $\m{p} = (p, p)$ after increasing the multiplicities of the internal meshlines by $s$ and by $2s$ those of the boundary meshlines.
\end{proposition}

\begin{proof}
Increasing the multiplicities as stated in the Proposition immediately satisfies the first constraint for an LR mesh to be admissible for RM B-splines of degree $\m{p}$. Since the original LR mesh is admissible for $\m{1}$, each cell of the mesh is covered by exactly $4$ RM B-splines. For degree $\m{1}$, the B-spline systems coincide with the RM B-splines themselves, meaning that each system contains a single function. Upon increasing the multiplicities and the degree from $\m{1}$ to $\m{p}$, the B-spline systems occupy the same regions as before, corresponding to the supports of the original bilinear RM B-splines (see \cite[Section 6.4]{rm1}). Each system now contains $(s + 1)^2$ functions of degree $\m{p}$, yet their support arrangement on the mesh remains unchanged. Therefore, each cell is still contained in $4$ B-spline systems, and the number of RM B-splines nonzero over that cell is $4(s + 1)^2 = (2s + 2)^2 = (p + 1)^2$, which matches the dimension of the bivariate polynomial space of degree $\m{p}$ over the cell. Hence, the RM B-splines are linearly independent on each cell, and thus locally linearly independent in the domain. This ensures that the second admissibility condition is satisfied as well.
\end{proof}

\begin{corollary}\label{cor}
If an LR mesh is admissible for RM B-splines of degree $\m{p} = (p, p)$ with $p = 2s + 1$, then for any other degree $\m{p}' = (p', p')$ with $p' = 2s' + 1$, the same mesh becomes admissible for RM B-splines of degree $\m{p}'$ by adjusting the meshline multiplicities from $s + 1$ to $s' + 1$ internally, and from $2s + 2 = p + 1$ to $2s' + 2 = p' + 1$ on the boundary.
\end{corollary}
\begin{proof}
The procedure outlined in Proposition \ref{prop} is reversible. Reducing the meshline multiplicities and the degree from $\m{p}$ to $\m{1}$ yields an LR mesh that is admissible for $\m{1}$. Applying Proposition \ref{prop} with respect to $\m{p}'$ then confirms that the mesh with adjusted multiplicities is admissible for RM B-splines of degree $\m{p}'$.
\end{proof}

When $s = 0$, i.e., $\m{p} = \pmb{1}$, the RM B-splines are defined on LR meshes with simple internal meshlines and open boundary with respect to $p$. Such meshes are the kind of meshes that are automatically generated by the Hierarchical LR (HLR) \cite{hlr}, Non-Nested Support Structured (N$_2$S$_2$) \cite{n2s2}, and Effective Grading (EG) \cite{eg} refinement strategies. These algorithms were developed to construct locally linearly independent LR B-splines for any degree $\m{p}$ on meshes with internal meshlines of multiplicity $1$. In the bilinear case, the resulting LR meshes are therefore admissible for bilinear RM B-splines. By Proposition \ref{prop}, the same meshes can be reused to construct RM B-splines of any degree $\m{p}$ simply by raising the meshline multiplicities.

Figure \ref{fig:ex} illustrates examples of N$_2$S$_2$, HLR, and EG refinements for bilinear LR B-splines, localized near three points, along the diagonal, and along a circular arc, i.e., some of the configurations already studied in the RM B-spline literature \cite{rm1, rm3, rm2}. RM B-splines of any degree can then be generated on these meshes by setting the meshline multiplicities to $s + 1$ in the interior and $2s + 2$ on the boundary. 
\begin{figure}
\centering
\subfloat[N$_2$S$_2$]{\includegraphics[width = 0.3\textwidth, page = 1]{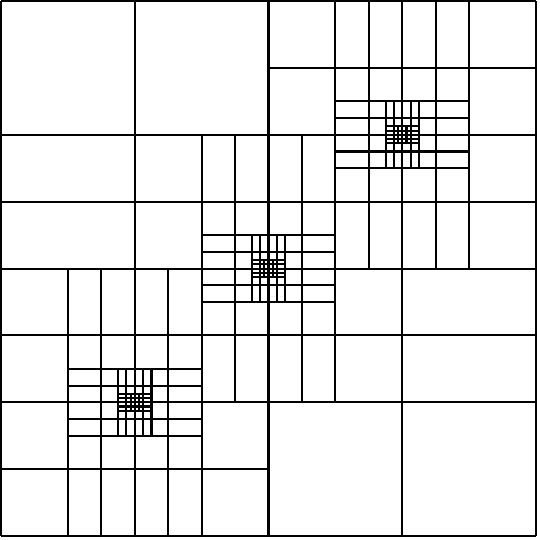}}\quad
\subfloat[HLR]{\includegraphics[width = 0.3\textwidth, page = 2]{ex}}\quad
\subfloat[EG]{\includegraphics[width = 0.3\textwidth, page = 3]{ex}}\\
\subfloat[N$_2$S$_2$]{\includegraphics[width = 0.3\textwidth, page = 4]{ex}}\quad
\subfloat[HLR]{\includegraphics[width = 0.3\textwidth, page = 5]{ex}}\quad
\subfloat[EG]{\includegraphics[width = 0.3\textwidth, page = 6]{ex}}\\
\subfloat[N$_2$S$_2$]{\includegraphics[width = 0.3\textwidth, page = 7]{ex}}\quad
\subfloat[HLR]{\includegraphics[width = 0.3\textwidth, page = 8]{ex}}\quad
\subfloat[EG]{\includegraphics[width = 0.3\textwidth, page = 9]{ex}}
\caption{Examples of LR meshes which are admissible for RM B-splines. The meshes on the left column are built using the N$_2$S$_2$, the central using the HLR and the right using the EG refinements, respectively, for bilinear LR B-splines. All the meshes have been constructed using 7 iterations starting from the open boundary as initial tensor mesh for $\m{p} = (1, 1)$. In the top row (a)--(c) we visually compare the refinements concentrated on three diagonal points. In the central row (d)--(f) we display the refinements along the diagonal. In the bottom row (g)--(i) we display the resulting meshes when refining along a circular arc. All the depicted LR meshes are admissible for RM B-splines of any degree by adjusting the meshline multiplicities as described in Proposition \ref{prop}.}
\label{fig:ex}       
\end{figure}

\section{Including RM B-splines in LR B-spline implementations}\label{sec:imp}
RM B-splines can be conveniently represented within the LR B-spline framework. It is sufficient to store the bilinear LR B-spline set on the mesh along with the desired smoothness parameter $s$. During refinement, there is no need to explicitly construct or store the RM B-splines, as all operations can be performed on the associated bilinear LR B-splines. Only for evaluation we need to construct RM B-splines, and even then, only those $4(s + 1)^2$ functions that are nonzero at the evaluation point are constructed. Specifically, given a point $\pmb{x}$ in the domain, we follow this procedure:
\begin{enumerate}
\item Identify the four bilinear LR B-splines that are nonzero at $\pmb{x}$.
\item Reconstruct the four corresponding B-spline systems of degree $\m{p}$.
\item Evaluate such local configuration using the de Boor--like algorithm described in \cite[Section 5]{rm2}.
\end{enumerate}
Hence, RM B-splines are not globally constructed or stored: only the few functions nonzero at the given point are built on-the-fly from the local knot vectors of the $4$ bilinear LR B-splines and evaluated.
This approach allows RM B-splines to be directly integrated into the existing LR B-spline framework and implementations. They can be generated dynamically when required and efficiently evaluated using the algorithm of \cite[Section 5]{rm2}.

Let us analyze more closely how to generate the RM B-splines in the four B-spline systems of step 2. Given a bilinear LR B-spline whose support contains the evaluation point $\pmb{x}$, we first increase the knot multiplicities in the local knot vectors of this bilinear LR B-spline, by $s$ or $2s$ according to Proposition \ref{prop}. The resulting (extended) knot vectors define a small tensor-product mesh, namely, the mesh underlying the B-spline system associated with that LR function. The $(s + 1)^2$ RM B-splines in the associated B-spline system are then simply the standard tensor-product B-splines built on this local tensor-product mesh. In this process, the extended knot vectors are therefore considered as global knot vectors for the system, and the rest of the LR mesh is completely dismissed. In other words, we isolate the small tensor-product mesh defined in the support of the bilinear LR B-spline, ignore all other local meshlines, and construct the associated B-spline system entirely within this self-contained tensor-product setting.

\section{Numerical considerations}\label{sec:num}
In this section we numerically compare RM B-splines with maximally smooth LR B-splines. 
Our goal is to quantify the effect of the reduced continuity of RM B-splines on both the dimension of the resulting spline spaces and the accuracy, in contrast to the maximal continuity of LR B-splines. 

A key observation is that, despite being defined on meshes with more localized meshlines, the weaker smoothness conditions imposed by RM B-splines leads to significantly larger spline spaces than those generated by maximally smooth LR B-splines. In practice, this implies that RM B-splines require solving much larger linear systems, while, as the numerical tests shall prove, this increase in dimension does not always yield a gain in approximation accuracy. 

As a first experiment, we compare the cardinalities of the RM B-spline set $\mathcal{R}$ with global regularity $C^s$, for $s \in \NN$, of bi-degree $\mathbf{p} = (p,p)$ where $p = 2s+1$, and the LR B-spline set $\mathcal{L}$ of bidegree $\m{p}$ with global regularity 
$C^{p-1}$, i.e., of maximal smoothness. The comparison is carried out on LR meshes constructed using the same refinement strategy, the same number of refinement levels, and identical regions of interest. In particular, we revisit the three refinement scenarios from Figure \ref{fig:ex}: the diagonal refinement, the three-point refinement along the diagonal, and the circular arc refinement. Figure \ref{fig:dim} reports the ratio
$$
\frac{\#\mathcal{R}}{\#\mathcal{L}}
$$
for $s \in \{0, \ldots, 10\}$, illustrating how the RM dimension can grow relative to its LR counterpart.
\begin{figure}
\centering
\subfloat[]{\includegraphics[width = 0.3\textwidth, page = 1]{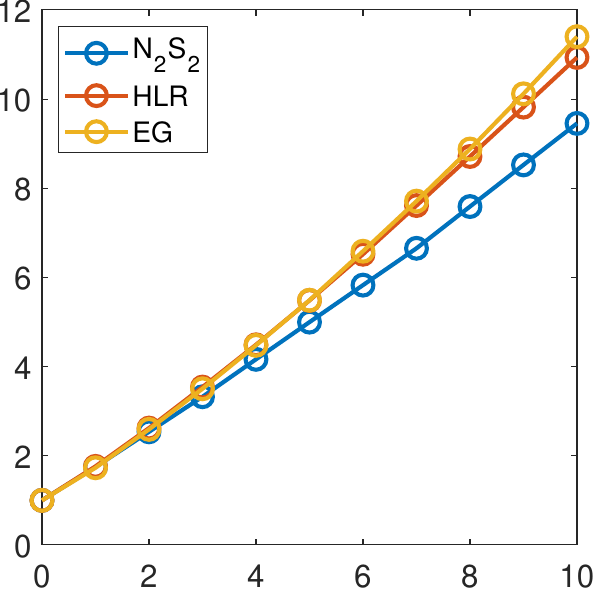}}\quad
\subfloat[]{\includegraphics[width = 0.3\textwidth, page = 2]{dimratio}}\quad
\subfloat[]{\includegraphics[width = 0.3\textwidth, page = 3]{dimratio}}
\caption{Comparisons of the cardinalities of sets of RM B-splines and LR B-splines of maximal smoothness, denoted by $\#\mathcal{R}$ and $\#\mathcal{L}$, of the same degree. The underlying LR meshes have meshlines (a) localized along the diagonal, (b) concentrated at three equidistant points on the diagonal, and (c) aligned along a circular arc, corresponding to the configurations shown in Figure \ref{fig:ex} for the RM B-splines. The meshes are generated using the $\mathrm{N}_2\mathrm{S}_2$, HLR, and EG refinement strategies. For RM B-splines, these strategies are applied to create LR meshes for LR B-splines of bidegree $(1,1)$, independently of the RM smoothness parameter $s$, as explained by Proposition \ref{prop}. For LR B-splines, the same refinement strategies are applied to build LR meshes for LR B-splines of bidegree $\mathbf{p} = (p,p)$ with $p = 2s+1$ and maximal smoothness.
The horizontal axis corresponds to $s \in \{0,\ldots,10\}$, while the vertical axis displays the ratio $\#\mathcal{R}/\#\mathcal{L}$.
}\label{fig:dim}
\end{figure}
Independently of the chosen refinement strategy, the numerical results exhibit two consistent trends:
\begin{itemize}
    \item The more extended the region in which refinement is applied,  the larger the discrepancy between the cardinalities of the RM and LR spline sets, regardless of the value of $s$ (and hence of $p$).
    
    \item As $s$ increases, the ratio $\#\mathcal{R} / \#\mathcal{L}$ can grow very rapidly. In some cases, the difference becomes dramatic, see for example the instance of the circular arc combined with the N$_2$S$_2$ refinement strategy in Figure \ref{fig:dim} (c).
\end{itemize}

The same trends persist, and become even more pronounced, in a second test, where we fix the smoothness $s \in \{0, \ldots, 10\}$ for both spaces. In this case, we compare the RM B-splines of degree $p = 2s+1$ and smoothness $C^{s}$ with LR B-splines of degree $p = s+1$, still of smoothness $C^{s}$. Figure \ref{fig:dims} reports the ratio $\#\mathcal{R} / \#\mathcal{L}$ for the same three refinement 
scenarios depicted in Figure \ref{fig:ex}. 

\begin{figure}
\centering
\subfloat[]{\includegraphics[width = 0.3\textwidth, page = 4]{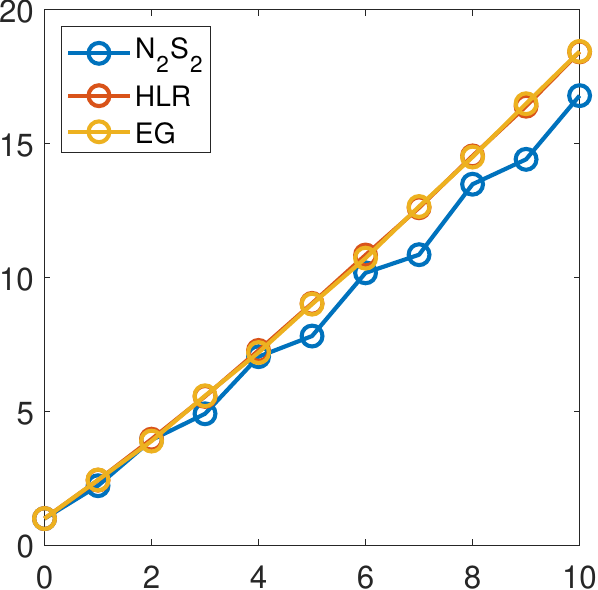}}\quad
\subfloat[]{\includegraphics[width = 0.3\textwidth, page = 5]{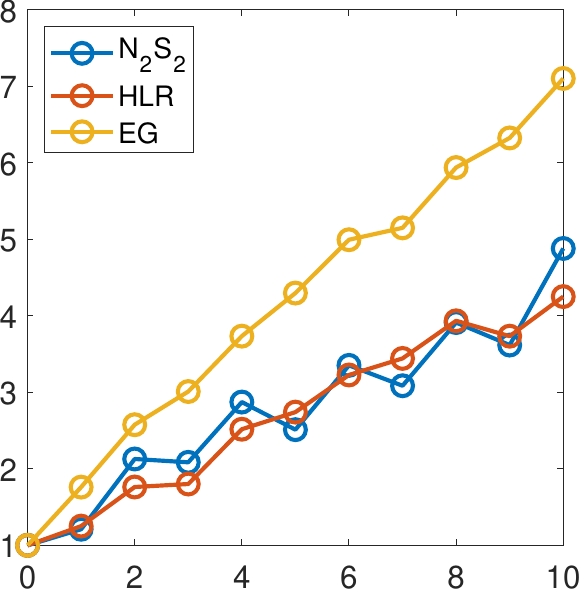}}\quad
\subfloat[]{\includegraphics[width = 0.3\textwidth, page = 6]{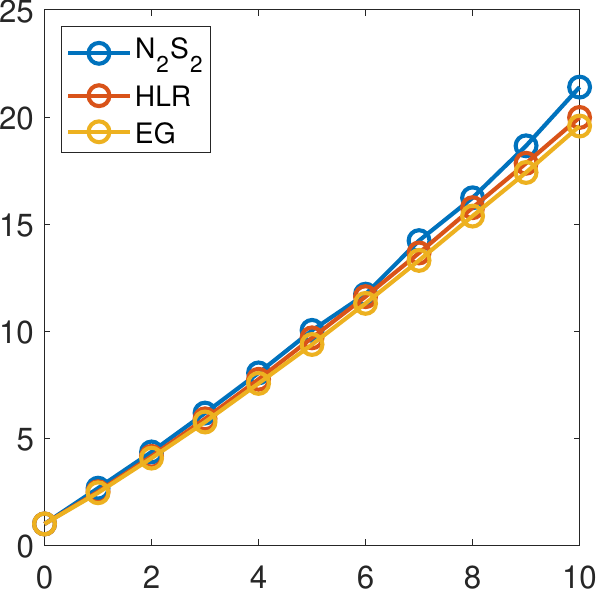}}
\caption{Comparisons of the cardinalities of the sets of RM B-splines and LR B-splines of maximal smoothness, denoted by $\#\mathcal{R}$ and $\#\mathcal{L}$, of same regularity. The underlying LR meshes have meshlines (a) localized along the diagonal, (b) concentrated at three equidistant points on the diagonal, and (c) aligned along a circular arc, corresponding to the configurations shown in Figure \ref{fig:ex} for the RM B-splines. The meshes are generated using the $\mathrm{N}_2\mathrm{S}_2$, HLR, and EG refinement strategies. For RM B-splines, these strategies are applied to create LR meshes for LR B-splines of bidegree $(1,1)$, independently of the RM smoothness parameter $s$, as explained in Proposition \ref{prop}. For LR B-splines, the same refinement strategies are applied to build LR meshes for LR B-splines of bidegree $(s + 1, s + 1)$ and maximal smoothness, i.e., $C^s$.
The horizontal axis corresponds to $s \in \{0,\ldots,10\}$, while the vertical axis displays the ratio $\#\mathcal{R}/\#\mathcal{L}$.
}\label{fig:dims}
\end{figure}


We finally compare the performances of $C^s$ RM B-splines of degree $2s + 1$ with respect to LR B-splines of maximal smoothness, still $C^s$ but with degree $s + 1$, within an automatically adaptive isogeometric refinement procedure.
To this end, we consider the Poisson problem with Dirichlet boundary conditions
\begin{equation}\label{eq:poisson}
\left\{
\begin{array}{lll}
-\Delta u(\pmb{x}) = f(\pmb{x}) 
& & \text{for }\pmb{x}\in \Omega \coloneqq [0,1]^2, \\\\
u(\pmb{x}) = u_D(\pmb{x}) 
& & \text{for }\pmb{x}\in \Gamma \coloneqq \partial\Omega ,
\end{array}
\right.
\end{equation}
whose analytical solution is
$$
u(\pmb{x}) = \arctan\left(100\left(\norm{\pmb{x} - \bar{\pmb{x}}}_2 - \frac{\pi}{3}\right)\right), \qquad \bar{\pmb{x}} \defeq (1.25, - 0.25)
$$
shown in Figure \ref{fig:poisson} (a). 
The source term $f$ and boundary values $u_D$ are defined accordingly.
For this experiment we fix $s = 2$ and perform seven iterations of the adaptive loop
\begin{equation}\label{eq:adaptiveIgA}
\scalebox{0.85}{
\begin{tikzpicture}[baseline={(current bounding box.center)}]
\matrix (m)[matrix of math nodes,column sep=5em,row sep=2em]{
 \dblock{SOLVE\,} \pgfmatrixnextcell \dblock{ESTIMATE\,} 
 \pgfmatrixnextcell \dblock{MARK\,}  
 \pgfmatrixnextcell \dblock{REFINE\,} \\
};
\draw[ultra thick, -stealth] (m-1-1) -- (m-1-2);
\draw[ultra thick, -stealth] (m-1-2) -- (m-1-3);
\draw[ultra thick, -stealth] (m-1-3) -- (m-1-4);
\draw[ultra thick, -stealth, rounded corners] 
 (m-1-4) -- (5.375, -1) -- (-5.45, -1) -- (m-1-1);
\end{tikzpicture}}
\end{equation}
starting from a uniform $8\times 8$ tensor mesh. 
After assembling the current approximation $u_h$ in the SOLVE step, we compute the element-wise $L^2$ errors 
$$
\|u_h - u\|_{L^2(\beta)}
$$
for each mesh cell $\beta$ during the ESTIMATE step.  In the MARK step we select all cells $\beta'$ satisfying
$$
\|u_h - u\|_{L^2(\beta')} 
\;\ge\; \theta\cdot 
\max_{\beta} \|u_h - u\|_{L^2(\beta)},
$$
with $\theta$ a percentage, that we have fixed to $0.05$. In the REFINE step the basis is enriched using the 
N$_2$S$_2$ strategy: all RM or LR B-splines whose support intersects at least one marked cell are refined.

The results are shown in Figures \ref{fig:poisson} (b)--(c). Figure (b) displays the global $L^2$ error decay for LR and RM B-splines, and Figure (c) shows the corresponding $L^\infty$ errors. The comparisons indicate that the higher degree of RM B-splines does \emph{not} yield improved accuracy per degree of freedom. 
\begin{figure}
\centering
\subfloat[]{\includegraphics[width = 0.375\textwidth, page = 1]{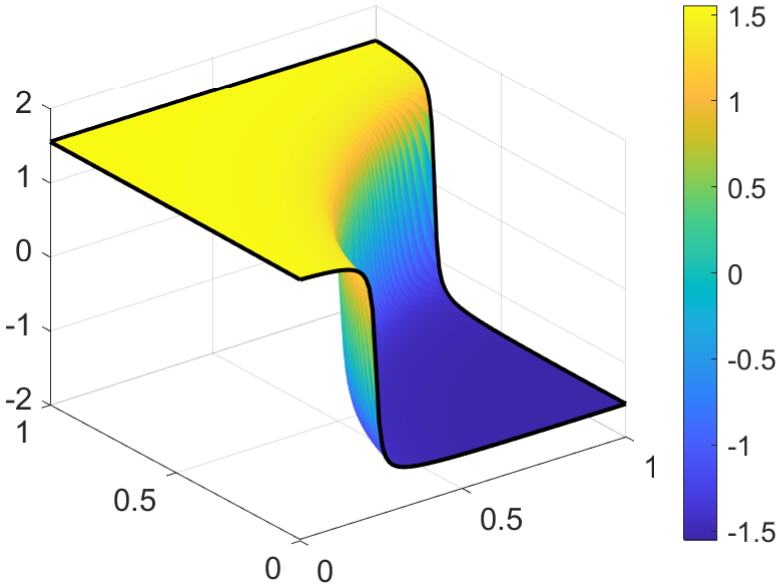}}\quad
\subfloat[]{\includegraphics[width = 0.275\textwidth, page = 2]{poisson}}\quad
\subfloat[]{\includegraphics[width = 0.275\textwidth, page = 3]{poisson}}
\caption{Comparisons of the $L^2$ and $L^\infty$ errors with LR and RM B-splines of the same smoothness ($C^2$) to approximate the solution of the Poisson problem \eqref{eq:poisson} in $[0, 1]^2$. In (a) the representation of the analytical solution in the domain. In (b) the $L^2$ error decays for LR and RM B-splines during 7 iterations of the adaptive IgA cycle \eqref{eq:adaptiveIgA} using the N$_2$S$_2$ refinement strategy, against the number of degrees of freedom in the $x$-axis. In (c) the $L^\infty$ decays.}\label{fig:poisson}
\end{figure}

It is also worth noting that the larger meshline multiplicities, characteristic of RM B-splines, produce visually sparser meshes than in the case of maximally smooth LR B-splines, even though the number of degrees of freedom is actually equal or smaller in the latter. Thus, mesh appearance may be misleading: RM meshes look coarser, yet they define the same or more basis functions than the respective LR meshes. This can be observed, for instance, in Figure \ref{fig:poissonmesh} (a)--(b), which show the final LR and RM meshes for the approximation considered for the Poisson problem \eqref{eq:poisson}.
\begin{figure}
\centering
\subfloat[]{\includegraphics[width = 0.4\textwidth, page = 4]{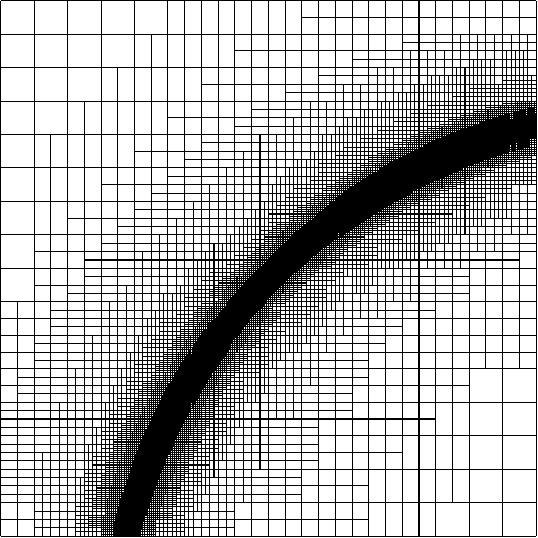}}\qquad
\subfloat[]{\includegraphics[width = 0.4\textwidth, page = 5]{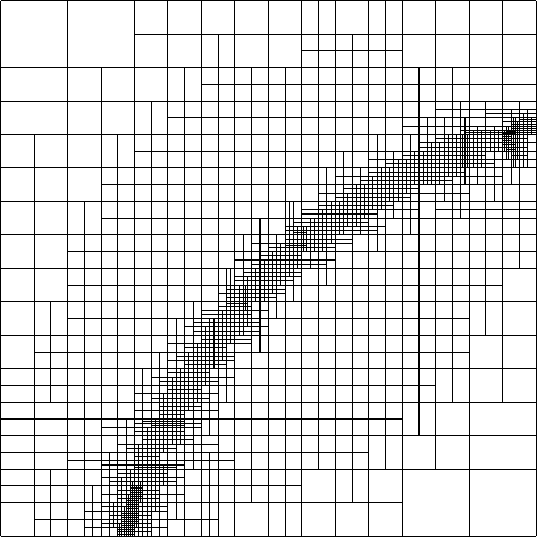}}
\caption{Comparisons of the LR meshes in the latest approximations of the solution of the Poisson problem \eqref{eq:poisson} using LR and RM B-splines of the same smoothness ($C^2$). In (a) the LR mesh for the LR B-splines of degree $s+1$, in (b) the LR mesh for the RM B-splines of degree $2s + 1$, with $s = 2$. Both meshes are obtained from an initial tensor mesh of $8\times 8$ boxes after 7 iterations of the adaptive IgA cycle \eqref{eq:adaptiveIgA} using the N$_2$S$_2$ refinement strategy. The mesh for the LR B-splines of maximal smoothness looks denser than that for the RM B-splines, despite the number of functions, of the respective type, is actually essentially the same.}\label{fig:poissonmesh}
\end{figure}

\section{Conclusion}\label{sec:con}
In this work we have reformulated RM B-splines within the LR B-spline framework. From this viewpoint, RM B-splines can be interpreted not as a distinct B-spline--like basis, but rather as a particular LR refinement mechanism that enforces local linear independence. The construction procedure for RM B-splines is, in fact, closely related to the N$_2$S property reinforcement step of the N$_2$S$_2$ strategy introduced in \cite{n2s2}, when applied to meshes whose meshlines have higher multiplicities.

This LR reinterpretation also makes transparent how admissible refinements for RM B-splines can be defined almost effortlessly by adapting established LR refinement strategies, namely N$_2$S$_2$, HLR 
and EG, originally designed for locally linearly independent LR B-splines on meshes with internal meshlines of multiplicity $1$. As shown in Proposition \ref{prop}, these strategies can be reused essentially unchanged, except for an increase in meshline multiplicities. Moreover, this perspective reveals that RM B-splines can be incorporated directly and with minimal effort into existing LR B-spline implementations.

Finally, our numerical investigations further show that the gap between degree and smoothness, intrinsic to RM B-splines, can lead to larger number of degrees of freedom to achieve similar accuracy. 
On the other hand, their simpler topological structure may still be attractive in contexts where efficient, specialized algorithms, such as fast evaluation procedures \cite{rm2}, can be exploited once the underlying approximation has been constructed.

\begin{acknowledgement}
This work has been supported by the MUR Excellence Department Project MatMod@TOV (CUP E83C23000330006) awarded to the Department of Mathematics of the University of Rome Tor Vergata. The author is member of the \textit{Gruppo Nazionale per il Calcolo Scientifico} of the Italian \textit{Istituto Nazionale di Alta Matematica “Francesco Severi”} (GNCS-INdAM). The INdAM support through GNCS 2025 project ``PASTRAMI - sPline And Solver innovaTions foR Adaptive isogeoMetric analysIs'' (CUP E53C24001950001) is gratefully acknowledged.
\end{acknowledgement}

\bibliographystyle{plain}
\bibliography{biblio}
\end{document}